\documentclass[12pt]{amsart}

\calclayout

\usepackage{amsmath, amssymb}
\usepackage{amscd}
\usepackage{verbatim}
\usepackage{bbold}

\makeatletter

\newtheorem{definition}{Definition}[section]
\newtheorem{theorem}[definition]{Theorem}

\theoremstyle{definition}
\newtheorem{remark}[definition]{Remark}

\newcommand{\noi}{\noindent}

\newcommand{\ra}{\rightarrow}

\newcommand{\mh}{\mu_{_H}}
\newcommand{\nh}{\nu_{_H}}
\newcommand{\su}{\mbox{supp}}
\newcommand{\sh}{{_H}}
\newcommand{\ep}{\varepsilon}

\begin{document}

\title[Topological Amenability of Semihypergroups]{Topological Amenability of Semihypergroups}

\author[C.~Bandyopadhyay]{Choiti Bandyopadhyay}

\address{Department of Mathematics \& Statistics, Indian Institute of Technology Kanpur, India.\\Current address: Department of Mathematics, SRM University AP, India.}

\email{choiti@ualberta.ca, choiti.b@srmap.edu.in}


\begin{abstract}
In this article, we introduce and explore the notion of topological amenability in the broad setting of (locally compact) semihypergroups. We acquire several stationary, ergodic and Banach algebraic characterizations of the same in terms of convergence of certain probability measures, total  variation of  convolution with probability measures and translation of certain functionals, as well as the F-algebraic properties of the associated measure algebra. We further investigate the interplay between restriction of convolution product and convolution of restrictions of measures on a sub-semihypergroup.  Finally, we discuss and characterize topological amenability of sub-semihypergroups in terms of certain invariance properties attained on the corresponding measure algebra of the parent semihypergroup. This in turn provides us with an affirmative answer to an open question posed by J. Wong in 1980. 
\end{abstract}

\maketitle

\section{Introduction}
\label{intro}

 The study of (topological) amenability on a given category of objects, have been one of the fundamental areas of research in abstract harmonic analysis. As a result, an extensive research has been carried out on different equivalent criteria and consequences  exhibited by  (topological) left/right amenable groups, semigroups and hypergroups (see \cite{ MI, DA, JP, SK, WO2,  WO0} for details). This article extends the concept of topological amenability to the broader class of (semitopological) semihypergroups, and attains several characterizations of the same in terms of certain stationary, ergodic, Banach algebraic and hereditary properties exhibited by the corresponding measure algebras. In addition, given a convolution of complex Radon measures on the underlying space, we explore the interplay between the restriction of convolution and convolution of restrictions on certain subsets, to eventually provide an estimate on how far apart these measures can be in terms of total variation.

 A semihypergroup is essentially a locally compact space, where the space of complex Radon measures admits a certain  associative convolution-algebra. Of course, as the name suggests, one can regard a semihypergroup as a natural extension to the classical theory of  locally compact semigroups, where the product of two points is a compact set, rather than a single point. The study of  (semi)hypergroups was initiated by Dunkl \cite{DU}, Jewett \cite{JE} and Spector \cite{SP} independently around 1970's and has seen substantial progress since then, as intriguing connections to natural structures arising from different fields of mathematics are being discovered till date \cite{BH}. 
 
 Other than locally compact semigroups, groups and hypergroups, some important examples of semihypergroups include certain coset spaces, homogeneous spaces and certain orbit spaces arising in different areas of research such as Lie groups,   dynamical systems, topological graph theory and partial differential equations, to name a few (see \cite{CB1, BH, JE} for examples explaining why these spaces fail to even have a hypergroup structure). The intrinsically different nature of the convolution algebra (and hence of the subsequent topology) when compared to the natural convolution of measures on a locally compact semigroup, prevents the existing classical results of semigroup theory from extending readily to this broader class of objects. We have organized the article as the following.

The next, i.e, second section of the article contains a basic set of preliminary notations and definitions that we would need for the rest of the article. In the third section, we first introduce topological invariant means and amenability in the category of semitopological semihypergroups, followed by a brief exposition of F-algebras and associated amenability concepts. The first main theorem (Theorem \ref{mainthm0}) explores several equivalent criteria for the existence of a topological left invariant mean (TLIM) on  a semihypergroup, in terms of certain stationary, ergodic and Banach algebraic properties.  

The rest of this section is dedicated to understanding the relation between a sub-semihypergroup and its parent semihypergroup, in terms of topological invariance as well as extensions and restrictions of certain measures and functions. Theorem \ref{convr} provides necessary insight into how the restrictions of convolution of measures interact with the convolutions of the restricted measures. This in turn provides us with the assertion  that these measures indeed coincide whenever the semihyergroup in question is a hypergroup. Using this phenomenon, in the final result of this section (Theorem \ref{thm2}) we acquire some useful equivalent criteria for a sub-semihypergroup $H$ to admit a TLIM, in terms of certain $H$-invariant means of the parent semihypergroup. This extends a result of Wong \cite{WO} even in the case of locally compact semitopological semigroups.





\section{Preliminary}
\label{Preliminary}

\noi Here we first list some of the preliminary set of basic notations that we will use throughout the text, followed by a brief introduction to the tools and concepts needed for the formal definition of a semihypergroup (referred to as `semiconvo' in \cite{JE}). All the topologies throughout this text are assumed to be Hausdorff, unless otherwise specified.

\vspace{0.03in}

\noi For any locally compact Hausdorff topological space $X$, we denote by $M(X)$ the space of all regular complex Borel measures on $X$, where $M^+(X), P(X)$ and $P_c(X)$ denote the subsets of $M(X)$ consisting of all finite non-negative regular Borel measures, the probability measures and probability measures with compact support on $X$, respectively. Moreover, $C(X)$ and $ C_c(X)$ denote the function spaces of all  bounded continuous and compactly supported continuous  functions  on $X$ respectively.

\vspace{0.03in}

\noi Next, we introduce two very important topologies on the positive measure space and the space of compact subsets for any locally compact topological space $X$. Unless mentioned otherwise, we will always assume these two topologies on the respective spaces.

\vspace{0.03in}

\noi The \textit{cone topology} on $M^+(X)$ is defined as the weakest topology on $M^+(X)$ for which the maps $ \mu \mapsto \int_X f \ d\mu$ is continuous for any $f \in C_c^+(X)\cup \{1_X\}$ where $1_X$ denotes the characteristic function of $X$. Note that if $X$ is compact then it follows immediately from the Riesz representation theorem that the cone topology coincides with the weak*-topology on $M^+(X)$ in this case.

\vspace{0.03in}

\noi We denote by $\mathfrak{C}(X)$ the set of all compact subsets of $X$. The \textit{Michael topology} \cite{MT} on $\mathfrak{C}(X)$ is defined to be the topology generated by the sub-basis $\{\mathcal{C}_U(V) : U,V \textrm{ are open sets in } X\}$, where
$$\mathcal{C}_U(V) = \{ C\in \mathfrak{C}(X) : C\cap U \neq \O, C\subset V\}.$$

\noi Note that $\mathfrak{C}(X)$ actually becomes a locally compact Hausdorff space with respect to this natural topology. 
 For any locally compact Hausdorff space $X$ and any element $x\in X$, we denote by $p_x$ the point-mass measure or the Dirac measure at the point $\{x\}$.

\vspace{0.03in}

\noi For any three locally compact Hausdorff spaces $X, Y, Z$, a bilinear map $\Psi : M(X) \times M(Y) \rightarrow M(Z)$ is called \textit{positive continuous} if the following holds :
\begin{enumerate}
\item $\Psi(\mu, \nu) \in M^+(Z)$ whenever $\mu \in M^+(X), \nu \in M^+(Y)$.
\item The map $\Psi |_{M^+(X) \times M^+(Y)}$ is continuous.
\end{enumerate}


\noi Now we are ready to state the formal definition for a semihypergroup. Note that we follow Jewett's notion \cite{JE} in terms of the definitions and notations, in most cases.


\begin{definition}\label{shyper}\textbf{(Semihypergroup)} A pair $(K,*)$ is called a (topological) semihypergroup if they satisfy the following properties:


\begin{description} 
\item[(A1)] $K$ is a locally compact Hausdorff space and $*$ defines a binary operation on $M(K)$ such that $(M(K), *)$ becomes an associative algebra.

\item[(A2)] The bilinear mapping $* : M(K) \times M(K) \rightarrow M(K)$ is positive continuous.

\item[(A3)] For any $x, y \in K$ the measure $p_x * p_y$ is a probability measure with compact support.

\item[(A4)] The map $(x, y) \mapsto \mbox{ supp}(p_x * p_y)$ from $K\times K$ into $\mathfrak{C}(K)$ is continuous.
\end{description}
\end{definition}


\noi For any $A,B \subset K$ the convolution of subsets is defined as the following:
$$A*B := \cup_{x\in A, y\in B} \ supp(p_x*p_y)  .$$
It is known \cite{JE, MT} that this convolution is a continuous binary operation on $\mathfrak{C}(K)$ and for any Borel subspace $E\subseteq K$, we have that  $\mathfrak{C}(E)$ is again a Borel subspace of $\mathfrak{C}(K)$, with $\mathfrak{C}(E)$ being closed (resp. compact) if $E$ is closed (resp. compact) in $K$.

A pair $(K, *)$ is called a \textit{semitopological semihypergroup} if if it satisfies all the conditions of Definition \ref{shyper} with property \textbf{(A$2$)} and \textbf{(A$4$)} replaced respectively by the following:
\begin{description}
\item[(A2$'$)] The map $(\mu, \nu) \mapsto \mu*\nu$ is positive and separately continuous on $M^+(K)\times M^+(K)$.
\item[(A4$'$)] The map $(x, y) \mapsto \mbox{ supp}(p_x * p_y)$ from $K\times K$ into $\mathfrak{C}(K)$ is separately continuous.
\end{description}

\noi Recall that for any locally compact Hausdorff space $X$, a homeomorphism $i : X \ra X$ is called a topological involution if $i\circ i(x) = x$ for each $x\in X$. For a semihypergroup $K$, a topological involution $i : K \ra K$ given by $i(x):= \tilde{x}$ is called a \textit{hypergroup involution} if for any $\mu, \nu \in M(K)$ we have that 
$(\mu * \nu)^{\widetilde{}} = \tilde{\nu} * \tilde{\mu}$, where for any measure $\omega \in M(K)$ we have that $\tilde{\omega}(B) := \omega(\tilde{B})= \omega(i(B))$ for any Borel measurable subset $B$ of $K$.


As expected, a hypergroup involution on a semihypergroup is analogous to the inverse function on a semigroup. Hence a semihypergroup with an identity and an involution of the following characteristic is a hypergroup.

\begin{definition}\textbf{(Hypergroup)} A pair $(H, *)$ is called a (topological)  hypergroup if it is a semihypergroup and satisfies the following conditions :

\begin{description}
\item[(A5)] There exists an element $e \in H$ such that $p_x * p_e = p_e * p_x = p_x$ for any $x\in H$.
\item[(A6)] There exists a hypergroup involution $x\mapsto \tilde{x}$ on $H$ such that $e \in \mbox{supp} (p_x * p_y)$ if and only if $x=\tilde{y}$.
\end{description}
\end{definition}

\noi The element $e$ in the above definition is called the \textit{identity} of $H$. Note that the identity and the involution are necessarily unique \cite{JE}.

\vspace{0.03in}

\begin{remark}
Given a Hausdorff topological space $K$, in order to define a continuous bilinear mapping $* : M(K) \times M(K) \rightarrow M(K)$, it suffices to only define the measures $(p_x*p_y)$ for each $x, y \in K$. This is true since we can then extend `$*$' linearly to $M_F^+(K)$. As $M_F^+(K)$ is dense in $M^+(K)$ \cite{JE}, we can further extend `$*$' to $M^+(K)$ and hence to the whole of $M(K)$ using linearity.
\end{remark}





\begin{remark} \label{remex}
As mentioned in the introduction above, important examples of semihypergroups which fail to be hypergroups, include coset spaces and certain orbit spaces with non-automorphic actions arising from locally compact groups (see \cite[Section 3]{CB1}, \cite{BH, JE} for details on  such natural structures as well as why most of these structures fail to retain a  hypergroup structure).  In particular, coset spaces such as $GL_n(\mathbb{R})/O_n(\mathbb{R}), GL_n(\mathbb{C})/SU_n(\mathbb{C}),\\ S_4/D_8$ or in general for any compact subgroup $H$ of a locally compact group $G$, the left coset space $K=G/H$ becomes a semihypergroup with the convolution given by the natural averaging technique:

$$p_{_{xH}} * p_{_{yH}} = \int_H p_{_{(xty)H}} \ d\mu(t) ,$$
where $\mu$ is the normalized Haar measure of $H$. More generally, given a continuous affine action $\pi$ of any compact group $H$ on a locally compact group $G$, the associated orbit space $K= G^H := \{x^H : x \in G\} $, where $x^H$ denotes the orbit of $x$, naturally becomes a  semihypergroup with the convolution given by the following standard  averaging technique:
$$ p_{_{x^H}} * p_{_{y^H}} := \int_H \int_H p_{_{(\pi(s, x)\pi(t, y))^H} }\ d\mu(s) \ d\mu(t) ,$$ where $\mu$ is the normalized Haar measure of $H$.
\end{remark}

For the rest of this article, unless otherwise specified, $K$ will denote a semitopological semihypergroup. A closed subset $H$ of $K$ is called a \textit{sub-semihypergroup} if it is closed under convolution, i.e, if $H*H\subseteq H$. Similarly, if $K$ is a hypergroup, then a closed subset $Q$ of $K$ is called a \textit{sub-hypergroup} if it is closed under convolution and hypergroup involution, i.e, if $Q*Q\subseteq Q$ and $\widetilde{Q}=Q$.

For any bounded measurable function $f$ on a (semitopological) semihypergroup $K$ and each $x, y \in K$ the left translate of $f$ by $x$ (resp. the right translate of $f$ by $y$), denoted as $L_xf$ (resp. $R_yf$) is defined as the following.
$$ L_xf(y) = R_yf(x) = f(x*y) = \int_K f \ d(p_x*p_y)\ . $$
The right centre of $K$, denoted as $Z_r(K)$, is defined as the following:
$$Z_r(K):= \{x\in K: \mbox{supp}(p_x* p_y) \mbox{ is singleton for each } y\in K  \}.$$
We can similarly define the left centre of a semihypergroup (see \cite{CB3, CB4, RO} for examples and detailed discussions on centres of semihypergroups as well as hypergroups). Given an element $x\in K$ and a Borel set $E\subseteq K$, we introduce the following two notions, adhereing to the conventional notions of `left inverse sets' in the theory of topological semigroups.
\begin{eqnarray*}
\tilde{x}E &:=& \{y\in K : (p_x * p_y)(E)>0 \},\\
x^{-1}E &:=& \{y\in K : (p_x * p_y)(E)=1 \}.
\end{eqnarray*}

Note that both these sets are Borel subsets of $K$, since $\tilde{x}E = \mbox{supp}(L_x\chi_{_E})$ and translates of Borel functions on a semihypergroup are Borel measurable as well \cite{JE}. Furthermore, $x^{-1}E= \tau^{-1}(\mathfrak{C}(E))$ where $\tau:K\ra \mathfrak{C}(E)$ is the continuous map \cite{JE, MT} given by $\tau(z):= \{x\}*\{z\}$. We follow the same convention and define `right inverse sets' on a semihypergroup $K$ in the following manner:
\begin{eqnarray*}
	E\tilde{x} &:=& \{y\in K : (p_y*p_x)(E)>0 \},\\
	Ex^{-1} &:=& \{y\in K : (p_y*p_x)(E)=1 \}.
\end{eqnarray*}

\begin{remark} \label{inrem}
	We immediately see that if $x\in Z_r(K)$ or $x\in E$ where $E$ is a right ideal, i.e, $E*K\subseteq E$, then we have that $\tilde{x}E= x^{-1}E$ (see \cite{CB1} for details on ideals of semihypergroups). In fact, if $K$ is a hypergroup, then we have that $\tilde{x}E\subseteq \{\tilde{x}\}*E$, since $(\{x\}*\{y\})\cap E \neq \O$ if and only if $y\in \{\tilde{x}\}*E$ \cite[Lemma 4.1B]{JE}. Moreover, if $E$ has non-empty interior, then we have that $\tilde{x}E= \{\tilde{x}\}*E$. Hence if $K$ is a hypergroup and $Q\subseteq K$ is any sub-hypergroup, then $\tilde{x}Q= x^{-1}Q=Q$ for any $x\in Q$, as $\{\tilde{x}\}*Q=Q$.
\end{remark}
\noi We say that a  subset $E$ of a semitopological semihypergroup $K$ is of \textit{Type L} (resp. of \textit{Type R}), if we have that $\tilde{x}E= x^{-1}E$ (resp. $E\tilde{x}= Ex^{-1}$) for each $x\in E$.

\section{Topological Left Amenability}
\label{main}

 In this section, we first introduce the concept of topological amenability in the context of a semihypergroup, and then proceed to discuss some hereditary properties of the same, as well as those regarding (classical) left amenability. 
 
 Whenever appropriate, we identify a bounded measurable function $f$ on $K$ with the function $\phi_f\in M(K)^*$ defined as $\phi_f(\mu)=\int_K f \ d\mu$. We know that $M(K)$ is an associative Banach algebra, and hence adhering to the Arens product structure (see \cite{AR, CB1, CB4} for details) on $M(K)^{**}$, for each $\mu\in P(K), \phi\in M(K)^*$,  we define the left translate of $\phi$ by $\mu$, denoted as $L_\mu\phi\in M(K)^*$, as $$L_\mu\phi(\nu) :=\phi(\mu*\nu)$$ for each $\nu\in M(K)$.  Note that no confusion in notations should arise here, when compared to the left translation $L_xf$ for a function $f$ on $K$, as it is  understood that for each $y\in K$, we have that 
  $$L_xf(y)= \phi_f(p_x*p_y) = L_{p_{_x}} \phi_f (p_y),$$ 
 \noi whenever all the quantities exist. In fact, it immediately follows \cite[pp 11]{JE} that for each $x\in K$ and  bounded measurable function $f$ on $K$, we have that $\phi_{L_xf}= L_{p_{_x}}\phi_f$, since for each $\nu\in M(K)$ we have that $$ \phi_{L_xf}(\nu) = \int_K\int_K f(z) \ d(p_{_x}*p_y)(z) \ d\nu(y) = \int_K f(z)\ d(p_x*\nu)(z) = L_{p_{_x}}\phi_f(\nu).$$
 
 We define the right translate $R_\mu\phi\in M(K)^*$ of $\phi\in M(K)^*$ by $\mu\in M(K)$ in a similar manner, as $$R_\mu\phi(\nu) := \phi(\nu*\mu),$$ for each $\nu\in M(K)$.  Recall that a  positive  linear functional  $m$ on $M(K)^*$  with norm $1$ is called a \textit{mean}, i.e, we have that $m\in M(K)^{**}$ such that $||m|| = 1$ and $m(\phi)\geq 0$ whenever $\phi\geq 0$ for any $\phi\in M(K)^*$.  We denote the set of all means on $M(K)^*$ as $\mathcal{M}(K)$. We know \cite{MI} that the following three formulations of a mean $m\in M^{**}(K)$ are equivalent:
 \begin{enumerate}
 	\item $m$ is positive and $||m||=1$.
 	\item $||m||=m(\mathbb{1})=1$,  where $\mathbb{1}$ denotes the  function $\chi_K \in M(K)^*$ given as $\mu\mapsto \mu(K)$.
 	\item $\inf\{\phi(\nu): \nu\in P(K)\} \leq m(\phi) \leq \sup\{\phi(\nu): \nu\in P(K)\}$ for each $\phi\in M(K)^*$.
 \end{enumerate}
 
 We say that a mean $m\in \mathcal{M}(K)$  is a  \textit{topological left-invariant mean} (TLIM) if $m(L_\nu \phi) = m(\phi)$ for any $\nu\in P(K)$, $\phi\in M(K)^*$.  Similarly, we say that a mean $m\in \mathcal{M}(K)$  is a  \textit{left-invariant mean} (LIM) if $m(L_{p_{_x}} \phi) = m(\phi)$ for any $x\in K$, $\phi\in M(K)^*$. A semitopological semihypergroup is called topological left amenable (resp. left amenable), if it admits a topological left-invariant mean (resp. left-invariant mean). 
 
 \begin{remark}
In some previous works \cite{CB1, CB4} we  defined the left amenability of a semitopological semihypergroup $K$ in terms of the existence of a left invariant mean on the function space $C(K)$. The notion of topological  amenability used in the present article in terms of means on $M(K)^*$ is a more general formulation of amenability, and readily implies amenability in terms of means on $C(K)$ via the natural inclusion $C(K)\subseteq M(K)^*$.
 \end{remark}
 
The following theorem offers several characterizations of topological amenability for a semihypergroup in terms of certain stationary, ergodic and Banach algebraic conditions. Certain equivalences for the category of locally compact semigroups, were proved by Wong in \cite{WO1}. We make use of similar techniques in certain parts of the proof.  Before stating the theorem, we need a quick detour through the basics of Banach algebra amenability and F-algebras \cite{LA0, JP}.

For a Banach algebra $A$, a Banach space $X$ is called a left Banach $A$-module if there is a bounded  bi-linear map $ (a, x) \mapsto a.x : A\times X \ra X$ such that $a_1.(a_2.x) = (a_1.a_2).x$ for each $a_1, a_2\in A, x\in X$. A right Banach $A$-module is defined likewise, along with an associative bounded bilinear map $ (x, a) \mapsto x.a : X\times A \ra X$. A Banach $A$-bimodule is a Banach space $X$ which is both left and right Banach $A$-module such that $$a_1.(x. a_2) = (a_1. x). a_2$$ for each $a_1, a_2\in A, x\in X$. Recall \cite{JP} that for every Banach $A$-bimodule $X$, we have that $X^*$ is naturally a Banach $A$-bimodule as well, via the following operations for each $a\in A, \phi\in X^*, x\in X$.
$$ (a.\phi) (x) := \phi(x.a), \ \ (\phi. a) (x) := \phi(a.x).$$ 
Given a Banach $A$-bimodule $X$, a derivation from $A$ into $X$ is a linear map $D:A\ra X$ such that $$D(a.b) = D(a).b + a. D(b)$$ for each $a, b\in A$. We immediately see that for each $x\in X$, the map $D_{x} (a):= a.x - x.a$ is a bounded derivation from $A$ into $X$. A derivation $D$ from $A$ into $X$ is called an \textit{inner derivation} if $D=D_x$ for some $x\in X$. We know that a  Banach algebra $A$ is called \textit{amenable} if for any Banach $A$-bimodule $X$ we have  that every bounded derivation of $A$ into $X^*$ is inner. We know that \cite{JO} a locally compact group $G$ is amenable if and only if $L^1(G)$ is amenable as a Banach algebra, although the implication fails to be true for the broader categories of locally compact semigroups and hypergroups  (see \cite{DA, JO, SK} for details on this subject).

 Recall that as introduced by  Lau in  \cite{LA0}, a Banach algebra $A$ is called an  F-algebra (referred to as  \textit{Lau algebra} by  Pier in \cite{JP}) if $A$ is the (unique) predual of a W$^*$-algebra $M$ such that the identity $e_M$ of $M$ is a multiplicative linear functional on $A$. We denote the predual of $M$ as $M_*=A$.
 
For any (semitopological) semihypergroup $K$, we know that $C(K)^*=M(K)$. Since $C(K)$ is a commutative C$^*$-algebra with the point-wise product, we immediately see that $M(K)^*$ is a W$^*$-algebra. In addition, since $(\mu*\nu)(K)=\mu(K)\nu(K)$ for any $\mu, \nu\in M(K)$, we have that $M(K)$ is indeed an F-algebra. An F-algebra $A$ with $A=M_*$ (i.e,  $A^*=M$) is called left amenable if for each Banach $A$-bimodule $X$ with the left module action given by $$a . x = 	\langle a, e_M\rangle x \ \ \mbox{ for each } \ \ a\in A, x\in X,$$ we have that every bounded derivation from $A$ into $X^*$ is inner.
 
 Finally, recall \cite{MI, WO1}  that the study of (topological) stationary properties is of pivotal importance in the study of amenability properties of locally compact groups and semigroups in general. We introduce the concept in a similar manner for the category of (semitopological) semihypergroups. We say that a (semitopological) semihypergroup $K$ is \textit{topological right stationary} if there exists a net $\{\nu_\alpha\}\subset P(K)$ such that for each $\phi\in M(K)^*$ there exists some scalar  $c_0\in \mathbb{C}$ such that $R_{\nu_\alpha} \phi \longrightarrow^{\hspace{-0.18in} w^*} \ c_0\mathbb{1}$ in  $M(K)^*$.
 

 \begin{theorem} \label{mainthm0}
 	Let $K$ be a semitopological semihypergroup. Then the following properties  are equivalent.
 	\begin{enumerate}
 		\item $K$ admits a TLIM.
 		\item $\exists$ a net $\{\nu_\alpha\}$ in $P(K)$ such that for each $\nu \in P(K)$, the net $\{(\nu * \nu_\alpha - \nu_\alpha)\}$ converges to $0$ in norm topology on $M(K)$.
 		\item For each $\mu\in M(K)$, we have $|\mu(K)|=\inf \{||\mu*\nu|| : \nu\in P(K)\}$.
 		\item $K$ is topological right stationary.
 		\item $M(K)$ is left amenable as an F-algebra.
 	\end{enumerate}
 	\end{theorem}
\begin{proof}
	$(2)\Rightarrow (1)$  is trivial by setting $m\in \mathcal{M}(K)$ to be the weak$^*$-limit of $\mu_\alpha$'s.
	
\noi $(1)\Rightarrow (2)$ follows in a similar manner as in the case for locally compact semigroups \cite{WO2} (a similar construction  was given in the setting of semihypergroups in \cite[Section 3]{CB4}).

\noi $(3)\Rightarrow (2)$: 	We first fix some $\nu_0\in P(K)$, and construct a net $\{\nu_\alpha\}_{\alpha\in J}\subseteq P(K)$ that satisfies the condition in $(2)$ for each $\nu\in P(K)$. For any $n\in \mathbb{N}$, here $P(K)^n$ denotes the product space $\{(\mu_1, \mu_2, \ldots, \mu_n): \mu_k\in P(K) \mbox{ for } 1\leq k \leq n\}$. We define the index set $J$ as the following. 
$$J:= \{\alpha: \alpha=(n_\alpha, P_\alpha, \varepsilon_\alpha) \mbox{ where } n_\alpha\in \mathbb{N}, P_\alpha = (\mu_1^\alpha, \mu_2^\alpha, \ldots, \mu_{n_\alpha}^\alpha) \in P(K)^{n_\alpha}, 0< \varepsilon_\alpha <1\}.$$ 
For any $\alpha, \beta\in J$ we say that $\alpha \geq \beta$ if $n_\alpha\geq n_\beta$, $P_\alpha \supseteq P_\beta$ and $\varepsilon_\alpha \leq \varepsilon_\beta$. It is easy to see that $(J, \leq)$ is a directed set since given $\alpha, \beta\in J$, we will always get an upper bound $\gamma=(n_\gamma, P_\gamma, \varepsilon_\gamma)\in J$ where $P_\gamma= P_\alpha\cup P_\beta$, $n_\gamma = |P_\gamma|$ and $\varepsilon_\gamma = \min(\varepsilon_\alpha, \varepsilon_\beta)$.

\noi Now pick any $\alpha\in J$. Since $P(K)*P(K)\subseteq P(K)$, for each $k= 1, 2, \ldots, n_\alpha$ we have that $(\mu_k^\alpha*\mu_0-\mu_0)(K)=0$, and hence for any $\omega\in P(K)$ we have that $$\big{(}(\mu_k^\alpha*\mu_0-\mu_0)*\omega\big{)}(K)=0.$$
 In particular, our assumption in $(3)$ implies that $$\inf \{||(\mu_1^\alpha*\mu_0-\mu_0)*\nu|| : \nu\in P(K)\}=0.$$ Pick $\nu_1^\alpha\in P(K)$ such that $|| (\mu_1^\alpha*\mu_0-\mu_0)*\nu_1^\alpha||\leq \ep_\alpha$. Next, since $\big{(}(\mu_2^\alpha*\mu_0-\mu_0)*\nu_1^\alpha\big{)}(K)=0$ as well, using the condition in $(3)$ again we get some $\nu_2^\alpha\in P(K)$ such that $|| (\mu_2^\alpha*\mu_0-\mu_0)*\nu_1^\alpha*\nu_2^\alpha||\leq \ep_\alpha$. Proceeding this way inductively, for each $1\leq k \leq n_\alpha$ we get some $\nu_k^\alpha\in P(K)$ such that
 $$ ||(\mu_k^\alpha *\mu_0-\mu_0)*\nu_1^\alpha*\nu_2^\alpha \cdots *\nu_k^\alpha||\leq \ep_\alpha.$$
 
\noi  For each $\alpha\in J$ we now set $\nu_\alpha:= \mu_0*\nu_1^\alpha*\nu_2^\alpha\cdots *\nu_{n_\alpha}^\alpha \in P(K)$. Given any $\nu\in P(K)$ and $\ep>0$, pick $\alpha_0\in J$ such that $\nu\in P_{\alpha_0}$  and  $\ep_{\alpha_0}<\ep$.   Then for each $\alpha\geq \alpha_0$ there exists some $j_\alpha\in \{1, 2, \ldots, n_{\alpha}\}$ such that $\nu = \mu_{j_\alpha}^\alpha \in P_\alpha \supseteq P_{\alpha_0}$. Hence we have that 
\begin{eqnarray*}
||\nu*\nu_\alpha -\nu_\alpha|| &=& || (\nu * \mu_0-\mu_0)*\nu_1^\alpha*\nu_2^\alpha\cdots *\nu_{n_\alpha}^\alpha||\\
&\leq & ||(\nu * \mu_0-\mu_0)*\nu_1^\alpha*\nu_2^\alpha\cdots * \nu_{j_\alpha}^\alpha|| \: ||\nu_{j_\alpha+1}^\alpha*\cdots * \nu_{n_\alpha}^\alpha||\\
&=& ||(\mu_{j_\alpha}^\alpha * \mu_0-\mu_0)*\nu_1^\alpha*\nu_2^\alpha\cdots * \nu_{j_\alpha}^\alpha|| \ \leq \ \ep_\alpha \leq \ep_{\alpha_0}<\ep,
\end{eqnarray*}
as required.	

$(5)\Leftrightarrow (1)$: It was shown in \cite[Theorem 4.1]{LA0} that an F-algebra $A$ is left amenable if and only if there exists an element $m\in P_1(A^{**}):= \{m\in A^{**}: m\geq 0, \ m(e_M) =1\}$ such that $m(x.\phi)=m(x)$ for each $\phi\in P_1(A^{**}) \cap A$, $x\in A^*$, where $x.\phi\in A^*$ is defined as $$(x.\phi)(\gamma):= x(\phi.\gamma),$$ for each $\gamma\in A$. In particular, letting $(A, \cdot)= (M(K), *)$ we immediately see that $P_1(A^{**})\cap A= P(K)$ and hence for each $\nu\in P(K)$, $\phi\in M(K)^*$  we have that $(\phi.\nu)=L_\nu\phi$, since  $$(\phi.\nu)(\mu) = \phi(\nu*\mu) = L_\nu\phi(\mu)$$ for each $\mu\in M(K)$. Thus we readily have that $M(K)$ is a left amenable  F-algebra if and only if $K$ admits a TLIM.

$(5)\Leftrightarrow (4) \Leftrightarrow (2)$: For any F-algebra $A$, for each $\phi\in P_1(A^{**}) \cap A$ and $x\in A^*$, the element $\phi . x \in A^*$ is defined as $$(\phi . x)(\gamma):= x(\gamma . \phi),$$ for each $\gamma\in A$. Hence again, setting $(A, \cdot)= (M(K), *)$ in \cite[Theorem 4.6]{LA0} and observing that $(\nu.\phi)=R_\nu\phi$ for each $\nu\in P(K)= P_1(A^{**})\cap A$, $\phi\in M(K)^*$, we have the desired implications.

$(4)\Rightarrow (3)$: Suppose that $K$ is topological right stationary. Pick any $\mu\in M(K)$ and set $\zeta_\mu:= \inf \{||\mu*\nu|| : \nu\in P(K)\}$. We first observe that $|\mu(K)| \leq \zeta_\mu$ since for each $\nu\in P(K)$ we have that $$ |\mu(K)| = |\mu(K)\nu(K)| = |(\mu*\nu)(K)| \leq |(\mu*\nu)|(K) = ||\mu*\nu||.$$
Now consider the convex set $S_\mu := \{\mu*\nu: \nu\in P(K)\}$. Since $S_\mu$ is bounded below by $\zeta_\mu$, by an application of the Hahn-Banach extension theorem, there exists  $\psi_\mu\in M(K)^*$ such that $||\psi_\mu||=1$ and $|\psi_\mu(\mu*\nu)|\geq \zeta_\mu$ for each $\nu\in P(K)$. For each $\nu, \omega \in P(K)$ we have that 
\begin{eqnarray*}
|R_\nu\psi_\mu (\mu*\omega)| &=& |\psi_\mu(\mu*\omega*\nu)|\\
&=& |\psi_\mu(\mu*(\omega*\nu))| \geq \zeta_\mu,	
\end{eqnarray*}
where the inequality follows since $\omega*\nu\in P(K)$. But $K$ is topological right stationary, and hence there exists a scalar $c_0\in \mathbb{C}$ such that $c_0\mathbb{1}$ lies in the weak$^*$-closure of the convex set $\{R_\nu\psi_\mu : \nu\in P(K)\} \subseteq M(K)^*$. In particular, we have that $|c_0\mathbb{1}(\mu*\omega)|\geq \zeta_\mu$ for each $\omega \in P(K)$. But $c_0\mathbb{1}\equiv c_0\mu(K)= c_0\mathbb{1}(\mu)$ on $S_\mu$. Hence in particular, we have the following set of inequalities.
\begin{eqnarray*}
\zeta_\mu \ \leq \	|c_0\mu(K)| &=& \inf \{|c_0\mathbb{1}(\mu*\omega)| : \omega\in P(K)\}\\
&\leq & |c_0| \inf\{||\mu*\omega||: \omega\in P(K)\}\\
&=& |c_0| \zeta_\mu.
	\end{eqnarray*}
Since $\zeta_\mu\neq 0$ we must have that $|c_0|=1$ and $$\zeta_\mu = |c_0|\: |\mu(K)| = |\mu(K)|,$$ as required.

	\end{proof}

Next, we proceed towards investigating the relation between the topological amenability of a semihypergroup and its sub-semihypergroups, in terms of the sub-semihypergroup invariance of a mean on $M(K)^*$. Let $H\subseteq K$ be a sub-semihypergroup of $K$. A mean $m\in \mathcal{M}(K)$ is called \textit{topological left $H$-invariant} if for any $\nu\in P(K)$ such that $\mbox{supp}(\nu)\subseteq H$ and  $\phi\in M(K)^*$ we have that $m(L_\nu\phi)=m(\phi)$. Similarly, a mean $m\in \mathcal{M}(K)$ is called \textit{left $H$-invariant} if for any $y\in H$  and  $\phi\in M(K)^*$ we have that $m(L_{p_{_y}}\phi)=m(\phi)$.

For any Borel set $E\subseteq K$ and $\mu\in M(K)$, we denote the restriction of $\mu$ on $E$ as $\mu_{_E}$, i.e, we have that $\mu_{_E}\in M(E)$ where $\mu_{_E}(B) = \mu(B)$ for any Borel set $B\subseteq E$. Similarly, the restriction $h|_{_E}$ of a Borel function $h\in \mathcal{B}(K)$ on $E$ is denoted by $h_{_E}$. On the other hand, given a measure $\mu\in M(E)$, we denote the natural extension of $\mu$ to $M(K)$ as $\mu^e$, i.e, we have that $\mu^e\in M(K)$ where $\mu^e(B) := \mu(B\cap E)$ for any Borel set $B\subseteq K$. Similarly, the extension of a Borel function $h\in \mathcal{B}(E)$ to $K$ is denoted by $h^e$, and is defined as $h^e(x)=h(x)$ if $x\in E$ and $h^e(x)=0$ for any $x\in E^c$. 

\begin{remark}\label{remext}
It follows readily from the definitions and notations above, that for any Borel set $E\subseteq K$ and a   measure $\omega\in M(K)$ such that $\su(\omega)\subseteq E$, we have that  $\omega_H^e=\omega$. Moreover, we have  that $M(E) = \{\mu_{_E} : \mu\in M(K)\}$ and that $$\int_E h^e \ d\mu = \int_K h^e \ d\mu = \int_E h \ d\mu_{_E} ,$$ for any $h\in \mathcal{B}(E)$, $\mu\in M(K)$, whenever the above quantities exist.
\end{remark}

The following theorem provides insight into how the convolution operation of the measure algebra $M(K)$ is distributed among the restrictions of measures to $M(H)$ for a sub-semihypergroup $H\subseteq K$.

\begin{theorem} \label{convr}
	Let $K$ be a semitopological semihypergroup and $H\subseteq K$ is a sub-semihypergroup. Then the following statements hold true for any $\mu, \nu\in M(K)$.
	\begin{enumerate}
		\item $||(\mu*\nu)_{_H} - (\mu_{_H} *\nu_{_H}) ||\leq \int_K |\mu|\big{(} (H\tilde{y}\cap H^c) \ d|\nu|(y)	 + 	\int_K|\nu|(\tilde{x}H\cap H^c)  \ d|\mu|(x)$.
		\item If $\mbox{supp}(\mu) \subseteq H$, then we have that $$||(\mu*\nu)_{_H} - (\mu_{_H} *\nu_{_H}) || \leq \int_H |\nu| \: (\tilde{y}H\cap H^c) \ d\mu(y) . $$
		\item If $\mbox{supp}(\mu), \mbox{supp}(\nu)\subseteq H$, then $(\mu*\nu)_{_H} = (\mu_{_H} *\nu_{_H})$.
		\item If $K$ is a hypergroup and $H$ a sub-hypergroup of $K$, then $(\mu*\nu)_{_H} = (\mu_{_H} *\nu_{_H})$ whenever $\su(\mu)\subseteq H$.
	\end{enumerate}
\end{theorem}

\begin{proof}
	
	
	First to prove $(1)$, pick and fix some  $f\in C(H)$. Although the functions $(R_yf)^e$ and $R_yf^e$ do not coincide everywhere on $K$, for any $x, y\in H$, we immediately have that $$(R_yf)^e(x) = R_yf(x) = \int_H f \ d(p_x*p_y) = \int_K f^e\ d(p_x*p_y) = R_yf^e(x). $$
Hence for each $y\in H$ we have the following.
\begin{eqnarray*}
\int_H f \ d(\mh*\nh) &=& \int_H\int_H f(x*y) \ d\mh(x) \ d\nh(y)\\
&=& \int_H\int_H R_yf(x) \ d\mh(x)\ d\nh(y)\\
&=&\int_H\int_H (R_yf)^e (x) \ d\mu(x)\ d\nh(y)\\ 	
&=& \int_H\int_H R_yf^e (x) \ d\mu(x)\ d\nh(y)\\ 	
&=& \int_H \int_H f^e(x*y) \ d\mu(x)\ d\nu(y).
\end{eqnarray*}	
Now consider the set  $R_H:= \{(x, y)\in (H^c\times K) \cup (K\times H^c): (p_x*p_y)(H)>0\}\subseteq K\times K$. Then for each $f\in C(H)$ we have the following inequality.

\begin{eqnarray*}
	\Big{|}\int_H f  d(\mu*\nu)_\sh - \int_H f  d(\mh*\nh)	\Big{|} &=& 	\Big{|}\int_K f^e  d(\mu*\nu) - \int_H \int_H f^e(x*y) d\mu(x) d\nu(y)	\Big{|}\\
	&=&	\Big{|}\int_K\int_K  f^e(x*y)   d\mu(x) d\nu(y) - \int_H \int_H f^e(x*y)  d\mu(x) d\nu(y)	\Big{|}\\
	&= & 	\Big{|}\iint_{{R_H}} f^e(x*y)   d\mu(x) d\nu(y)	\Big{|}\\
	&\leq & \iint_{{R_H}} \int_K \big{|}f^e(z)\big{|} d(p_x*p_y)(z)   d|\mu|(x) d|\nu|(y)\\
	&\leq & ||f||_\infty \iint_{R_H} 	d|\mu|(x) d|\nu|(y).
\end{eqnarray*}	
For any $y\in K$ if $(x, y) \in 	(H^c\times K)\cap {R_H}$, then following the notations introduced in previous section, we must have that $x\in H\tilde{y}$. Similarly, for any $x\in K$ if $(x, y) \in (K\times H^c)\cap {R_H}$, then we have that $y\in \tilde{x}H$ as well. Hence finally we have that

\begin{eqnarray*}
&&	\hspace{-0.25in} \Big{|}\int_H f  d(\mu*\nu)_\sh - \int_H f  d(\mh*\nh)	\Big{|} \\
&\leq & ||f||_\infty \Big{(}	\iint_{(H^c\times K)\cap {R_H}}  d|\mu|(x) d|\nu|(y)	 + 	\iint_{(K\times H^c)\cap {R_H}}  d|\nu|(y) d|\mu|(x) \Big{)}\\
	&= &  ||f||_\infty \Big{(}	\int_K\int_{(H\tilde{y}\cap H^c)}  d|\mu|(x) d|\nu|(y)	 + 	\int_K\int_{(\tilde{x}H\cap H^c)}   d|\nu|(y) d|\mu|(x) \Big{)}\\
	&= & ||f||_\infty \Big{(}	\int_K |\mu|\big{(} (H\tilde{y}\cap H^c) d|\nu|(y)	 + 	\int_K|\nu|(\tilde{x}H\cap H^c)   d|\mu|(x) \Big{)},\\
\end{eqnarray*}
as required.

The statements $(2)$ and $(3)$ follow immediately from $(1)$. To prove $(4)$, notice that it follows immediately from Remark \ref{inrem} that for each $y\in H$, we have that $\tilde{y}H = \{\tilde{y}\}*H= H$. Hence  $\tilde{y}H\cap H^c = \O$ for each  $y\in \su(\mu)$, and the conclusion follows from $(2)$.
	\end{proof}

We now prove the following equivalence between the existence of a TLIM on a semihypergroup and one of its sub-semihypergroups. The first two equivalences were proved in the setting of  locally compact semigroups in \cite{WO}. We follow similar ideas in proving certain sections of the following theorem. In addition to the intrinsically different nature of the measure algebra constructions between semihypergroups and locally compact (semi)groups, one of the main challenges here is the fact that for a given sub-semihypergroup  $H$ of $K$ and $x\in K$, we do not necessarily have that $L_x \chi_{_H }= \chi_{_{\tilde{x}H}}$, unlike locally compact semigroups, and even hypergroups when $x\in H$.

\begin{theorem}\label{thm2}
	Let $K$ be a semitopological semihypergroup and $H\subseteq K$ is a sub-semihypergroup. Then the following properties are equivalent.
	\begin{enumerate}
		\item $H$ admits a TLIM.
		\item $\exists$ a topological left $H$-invariant mean  $m\in \mathcal{M}(K)$  such that $m(\chi_{_H}) =1$.
		\end{enumerate}
	Moreover, if $H$ is of Type L, then the above properties are equivalent to the following:
	\begin{enumerate}
		\setcounter{enumi}{2}
		\item $\exists$ a topological left $H$-invariant mean $m\in \mathcal{M}(K)$ such that $m(\chi_{_H}) >0$.
	\end{enumerate}
\end{theorem}
\begin{proof}
	
	$(1)\Rightarrow (2)$ : Let $m_0\in M(H)^{**}$ be a TLIM of $H$. Consider the restriction map $r_H: M(K)^* \ra M(H)^*$ defined as $$r_H(\phi)(\mu):= \phi(\mu^e),$$ for each $\phi\in M(K)^*, \ \mu\in M(H)$. Since $\mu^e\geq 0$ whenever $\mu\geq 0$ in $M(H)$, we immediately see that $r_H$ is a positive linear map on $M(K)^*$. The natural extension of $m_0$ to $M(K)^{**}$ is the functional $m\in M(K)^{**}$ defined as $m:= m_0\circ r_H$.
	
\noi	Now observe that for each  $\mu\in P(K)$ where $\su(\mu)  \subseteq H$,  we have that $	r_H(L_\mu\phi) =  L_{\mu_{_H}} r_H(\phi)$, since for each $\nu \in M(H)$ and  $\phi\in M(K)^*$ we have the following equality.
	\begin{eqnarray*}
	r_H(L_\mu\phi)	(\nu) &=& L_\mu\phi(\nu^e)\\
	&=& \phi(\mu*\nu^e)\\
	&=& \phi((\mu*\nu^e)_H^e)\\
	&=& r_H(\phi) (\mu*\nu^e)_H\\
	&=& r_H(\phi) (\mu_H *\nu_H^e) \ = \  r_H(\phi) (\mu_H *\nu) \ =  L_{\mu_{_H}} r_H(\phi) (\nu),
	\end{eqnarray*}
where	the third and sixth equalities follow  from Remark \ref{remext} and since $$\su(\mu*\nu^e) =  \overline{\su(\mu) * \su(\nu^e)} \subseteq  \overline{H*H} \subseteq H$$ and the fifth equality follows from Theorem \ref{convr}. Hence for any $\nu\in P(K)$ such that $\su(\nu)\subseteq H$ and $\phi\in M(K)^*$ we have that
\begin{eqnarray*}
	m(L_\nu\phi) \ = \ m_0\circ r_H(L_\nu\phi) &=&  m_0\big{(} L_{\nu_{_H}} r_H(\phi)\big{)}\\
	&=& m_0\big{(} r_H(\phi)\big{)} \ = \ m(\phi),
	\end{eqnarray*}
where the last equality follows since $m_0$ is a TLIM. Thus we see that $m$ is topologically $H$-invariant. Moreover, since  $r_H$ is a positive linear map, we have that $m$ is a positive linear functional on $M(K)^*$ as well, such that $||m||\leq 1$. Finally,  since $$m(\chi_{_H}) = m_0\circ r_H(\chi_H)  = m_0(\chi_{_H}) = 1,$$ we see that $m$ is indeed a mean on $K$, as needed.  

\vspace{0.03in}

\noi $(2)\Rightarrow (1)$ :  On the other hand, let $m$ be a topological left  $H$-invariant mean in $\mathcal{M}(K)$ such that $m(\chi_{_H})
=1$.  Consider the extension map $e_H: M(H)^* \ra M(K)^*$ defined as $$e_H(\phi)(\mu):= \phi(\mh),$$ for each $\phi\in M(H)^*, \ \mu\in M(K)$.  Hence the natural restriction of $m$ to $M(H)^{**}$ is the functional $m_0\in M(H)^{**}$ defined as $m_0:= m \circ e_H$. Similarly as before, we immediately see that $e_H$ is a positive linear map on $M(K)^*$, and hence $m_0\in \mathcal{M}(H)$.

Now to see that $m_0$ is indeed topological left invariant, first note that  for each $\phi\in M(H)^*$,  $\nu\in P(H)$ and $\mu\in M^+(K)$ we have the following:
\begin{eqnarray*}
	\big{|} e_H(L_\nu\phi) (\mu) - L_{{\nu^e}} e_H(\phi) (\mu)\big{|} &=& 	\big{|} L_\nu\phi (\mh) -  e_H(\phi) (\nu^e*\mu)\big{|}\\
	&=&  	\big{|} \phi (\nu * \mh) -  \phi \big{(}(\nu^e*\mu)_{_H}\big{)}\big{|}\\
	&\leq &  ||\phi|| \:  ||((\nu^e)_{_H} * \mh) - (\nu^e*\mu)_{_H}||\\
	&\leq  & ||\phi|| \:  |\mu|(H^c) \:  \nu^e(H) \ = \ ||\phi|| \: \mu(H^c) \ = \ ||\phi||\: \chi_{_{H^c}} (\mu),
	\end{eqnarray*}
where the first inequality follows since $\su(\nu^e) \subseteq H$ and the second inequality follows from Theorem \ref{convr}. Hence for each $\phi\in M(H)^*$ and $\nu\in P(H)$ we have the following invariance, as required.
\begin{eqnarray*}
\big{|} m_0(L_\nu\phi) - m_0 (\phi)\big{|} &=& 	\big{|} m(e_H(L_\nu\phi)) - m (e_H(\phi))\big{|}\\
&=& 	\big{|} m\big{(}e_H(L_\nu\phi) - L_{\nu^e}e_H(\phi)\big{)}\big{|}\\
&\leq & ||\phi|| \: m\big{(}\chi_{_{H^c}}\big{)} \ = \ 0,
\end{eqnarray*}
where the second equality follows since $m\in \mathcal{M}(K)$ is topological $H$-invariant, the inequality follows since $m$ is a positive linear functional and the last equality holds true since we have that $$ m(\chi_{_{H^c}}) = m(\chi_{_{K}} - \chi_{_{H}}) = m(\chi_{_{K}}) - m(\chi_{_{H}}) = 0.$$

\vspace{0.03in}

\noi Now we further assume that $H$ is of Type L, i.e, $\tilde{x}H= x^{-1}H$  for each $x\in H$. Since it is sufficient to show that $(3)\Rightarrow (1)$, let $m$ be a topological left  $H$-invariant mean in $\mathcal{M}(K)$ such that $m(\chi_{_H})=:\alpha>0$. Set $m_0:= \alpha^{-1} (m \circ e_H)$. It then follows readily as above that $m_0\in \mathcal{M}(H)$. Now proceeding as above to check topological invariance of  $m_0$, we first note that for each $\phi\in M(H)^*$,  $\nu\in P(H)$ and $\mu\in M^+(K)$ we have the following:
\begin{eqnarray*}
	\big{|} e_H(L_\nu\phi) (\mu) - L_{{\nu^e}} e_H(\phi) (\mu)\big{|} &=& 	\big{|} L_\nu\phi (\mh) -  e_H(\phi) (\nu^e*\mu)\big{|}\\
	&=&  	\big{|} \phi (\nu * \mh) -  \phi \big{(}(\nu^e*\mu)_{_H}\big{)}\big{|}\\
	&\leq &  ||\phi|| \:  ||((\nu^e)_{_H} * \mh) - (\nu^e*\mu)_{_H}||\\
	&\leq  &  ||\phi|| \int_H \mu(\tilde{y}H\cap H^c) \ d\nu^e(y)\\
	&=& ||\phi|| \int_H \Big{(} \mu(\tilde{y}H\setminus H) \Big{)} \ d\nu^e(y)\\
	&=& ||\phi||  \Big{(}\int_H \mu(\tilde{y}H)  \ d\nu^e(y) - \mu(H)\Big{)}\\
	&=& ||\phi||  \Big{(}\int_H \int_K \chi_{_{\tilde{y}H}}(z) \ d\mu(z)  \ d\nu^e(y) - \mu(H)\Big{)}\\
	&=& ||\phi||  \Big{(}\int_H \int_K \chi_{H}(y*z) \ d\mu(z)  \ d\nu^e(y) - \mu(H)\Big{)}\\
	&=& ||\phi||  \Big{(}\int_K \int_K \chi_{H}(y*z) \  d\nu^e(y) \ d\mu(z)   - \mu(H)\Big{)}\\
	&=& ||\phi||  \big{(} \chi_{H}(\nu^e*\mu) - \mu(H)\big{)} \ = \ ||\phi||  \big{(} L_{\nu^e}\chi_{H} - \chi_H\big{)}(\mu),
\end{eqnarray*}
where the third equality follows since $H\subseteq \tilde{y}H$ as $y\in H$ and $H$ is a sub-semihypergroup, the sixth equality holds true since $H$ is of Type L and the second inequality follows from Theorem \ref{convr}. Hence for each $\phi\in M(H)^*$ and $\nu\in P(H)$ we have the following invariance, as required.
\begin{eqnarray*}
	\big{|} m_0(L_\nu\phi) - m_0 (\phi)\big{|} &=& 	\alpha^{-1} \big{|} m(e_H(L_\nu\phi)) - m (e_H(\phi))\big{|}\\
	&=& 	\alpha^{-1}	\big{|} m\big{(}e_H(L_\nu\phi) - L_{\nu^e}e_H(\phi)\big{)}\big{|}\\
	&\leq & 	\alpha^{-1} ||\phi|| \: m\big{(}L_{\nu^e}\chi_\sh - \chi_\sh\big{)} \ = \ 0,
\end{eqnarray*}
where the second and last equalities follow since $m\in \mathcal{M}(K)$ is topological $H$-invariant and  the inequality follows since $m$ is a positive linear functional on $M(K)^*$.
	\end{proof}

\begin{remark}

In particular, since any sub-semigroup of a locally compact semigroup is of Type L, the above theorem immediately provides an affirmative answer to the open question raised in \cite{WO}  regarding the equivalence of the statements $(2)$ and $(3)$ in the above theorem, in the setting of locally compact (semitopological) semigroups. Moreover, Theorems \ref{mainthm0}, \ref{convr} and \ref{thm2} potentially provide us with deeper insight into the convolution of measures on different coset, orbit and homogeneous spaces in terms of the respective  structures outlined  in Remark \ref{remex}.  

\end{remark}

\begin{remark}
	Since any commutative F-algebra is amenable \cite{LA0}, we immediately see that commutative semihypergroups are topological left amenable. In addition, it follows trivially that any non-commutative semihypergroup $(K, *)$ where the convolution is given as $p_x*p_y=p_y$ for each $x, y \in K$, is topological left amenable .
	
	As mentioned at the beginning of this section, although topological left amenability of a  semihypergroup immediately implies left amenability, it is still an open question whether the converse is true for general semitopological semihypergroups.  Moreover, although Theorem \ref{thm2} extends the results of \cite{WO} even for the case of locally compact (semitopological) semigroups, it is not known whether in the setting of a semihypergroup, the restriction that $H$ be a \textit{Type L} sub-semihypergroup, can be lifted.
\end{remark}

\section*{Acknowledgement}

\noi The author would  like to gratefully acknowledge the financial support provided by the Indian Institute of Technology Kanpur, India and Harish-Chandra Research Institute, India, where initial phases of the work were done. She is also grateful to the reviewers for the valuable comments and suggestions which led to a better representation of the article.

%


%



\end{document}